\documentclass[oneside,reqno]{amsart}
\usepackage[a4paper]{geometry}
\usepackage{mathrsfs}
\usepackage{amsmath,amssymb}\allowdisplaybreaks
\usepackage{mathdots}
\usepackage{mathtools}
\usepackage{bm}
\usepackage{graphicx} 
\usepackage{xcolor}
\usepackage{subfigure}
\usepackage{hyperref}
\hypersetup{
	colorlinks=true,
	citecolor=blue,
    linkcolor=blue,
    filecolor=magenta,      
    urlcolor=cyan,
}

\newtheorem{theorem}{Theorem}
\newtheorem{lemma}[theorem]{Lemma}
\newtheorem*{theorem1}{Theorem \ref{thm:1}}
\newtheorem*{theorem2}{Theorem \ref{thm:2}}
\newtheorem{proposition}[theorem]{Proposition}
\newtheorem{corollary}[theorem]{Corollary}
\newtheorem{conjecture}[theorem]{Conjecture}

\title{Self-similarity of $\mathcal{P}$-positions of $(2n+1)$-dimensional Wythoff's game}

\author{Yanxi Li}
\address[Y.-X. Li]{School of Mathematics, South China University of Technology, Guangzhou 510640, China}
\email{ma\_lyx@mail.scut.edu.cn}
\author{Wen Wu$^*$}\thanks{$^*$ Wen Wu is the corresponding author.}
\address[W. Wu]{School of Mathematics, South China University of Technology, Guangzhou 510640, China}
\email[corresponding author]{wuwen@scut.edu.cn}

\begin{document}
\begin{abstract}
	Wythoff's game as a classic combinatorial game has been well studied. In this paper, we focus on $(2n+1)$-dimensional Wythoff's game; that is the Wythoff's game with $(2n+1)$ heaps. We characterize their $\mathcal{P}$-positions explicitly and show that they have self-similar structures. In particular, the set of all $\mathcal{P}$-positions of $3$-dimensional Wythoff's game generates the well-known fractal set---the Sierpinski sponge. 
\end{abstract}
\maketitle

\section{Introduction}
The classic Wythoff's game is a two-player combinatorial game proposed by Wythoff \cite{Wythoff1907} in 1907. It is a variation of Nim's game, consisting of two players and two heaps of tokens. Two players take turns moving the tokens in one of two following ways: (i) taking a positive number of tokens from one heap; (ii) taking the same positive number of tokens from both heaps. 
The game ends when there are no tokens to remove. The player who makes the last move wins.

A game position is a tuple $(x_1,x_2)$ whose components are the numbers of tokens of two heaps. A position is called an $\mathcal{N}$-\emph{position} if the first player (i.e., the player about to move from there) has a winning strategy. A position is called a $\mathcal{P}$-\emph{position} if no matter what move the first player takes, the second player (i.e., the player who will play on the next round) always has a winning strategy. In the classic Wythoff's game, for instance, we can easily find that $(2,2)$ is an $\mathcal{N}$-position and $(1,2)$ is a $\mathcal{P}$-position. In general, a game position in the classic Wythoff's game is either an $\mathcal{N}$-position or a $\mathcal{P}$-position; see \cite{BCG01}. In order to win the game, the player should try to make sure the position after moving is a $\mathcal{P}$-position. Thus, characterizing $\mathcal{P}$-positions is the most important issue in the study of Wythoff's game.


The study of Wythoff's game can be divided into two categories: restrictions and extensions. In the restrictions of Wythoff's game, new rules are added to the original one and the players have fewer choices of movements. Ho \cite{HoTwo} added the rule that removing tokens from the smaller heap is not allowed if the two entries are not equal.  Duch\^ene and Gravier  \cite{Duch2009Geometrical} introduced the restriction that one cannot remove more than $R$ tokens from a single heap; meanwhile, Liu, Li and Li \cite{Liu2011A} allowed removing the same (arbitrarily large) number of tokens from both heaps. 
In \cite{Duch2009Geometrical}, Duch\^ene and Gravier also proposed the $(a,a)$ game; that is, a player may remove as many tokens from one heap or the other, or $k$ tokens from each, where $k$ is a positive multiple of $a$.
Aggarwal et al. \cite{AGSY16} studied the algorithm for $(2^b,2^b)$ game. Duch\^ene et al. \cite{DFNR10} investigated extensions and restrictions of Wythoff's game having exactly the same set of $\mathcal{P}$-positions as the original game.

The extensions of Wythoff's game also contain two parts: extensions of the moving methods, and extensions of the number of heaps. Fraenkel  \cite{Fraenkel1982How,Fraenkel1983Wythoff} studied the $a$-Wythoff's game by reducing the constraint on moving tokens, which was solved in both normal play and mis\`ere play. Gurvich \cite{Gurvich2012Further} investigated a more general case, called  the $WYT(a,b)$ game. In terms of the quantity of heaps, Fraenkel and Zusman \cite{Aviezri2001A} and Fraenkel  \cite{Fraenkel2004Complexity} proposed two different kinds of rules with $n$ heaps of tokens, and Duch\^ene  and Gravier \cite{Duch2009Geometrical} proposed the $n$ vectors game which further generalizes Wythoff's game.

In this paper, we focus on the Wythoff's game consisting of two players and $n$ heaps of tokens, called the \emph{$n$-dimensional Wythoff's game}. The game position is denoted by the $n$-tuple $(x_1,x_2,\dots,x_n)$ where the components are the number of tokens in the heaps. Let $V\subset\mathbb{N}^{n}\backslash\{\mathbf{0}\}$ be the set of \emph{move vectors}. A legal move $m$ is of the form $kv$ where $v\in V$ and $k\ge 1$ is an integer. 
Two players take turns to move. Namely, select a legal move $m=(m_1,m_2,\dots,m_n)$ and remove $m_i$ tokens from the $i$-th heap for all $i$. The player in turn who does not have any legal move loses the game. Note that in this case, the game may end although the heaps are not empty. The classic Wythoff's game is a $2$-dimensional Wythoff's game with move vectors $\{(1,0),(0,1),(1,1)\}$. The $\mathcal{P}$-position of the classic Wythoff's game can be described algebraically using complementary Beatty sequences as in Theorem \ref{thm:0} below. For the study of finding variations of Wythoff's game with prescribed complementary (Beatty) sequences as the set of $\mathcal{P}$-positions, one can see \cite{CDR16, DR10, LHF11} and reference therein. For more information on Beatty sequence, one can see for example \cite{AD19}.

\begin{theorem}[Wythoff \cite{Wythoff1907}]\label{thm:0}
	A position $(x_1,x_2)$ in the classic Wythoff's game is a $\mathcal{P}$-position if and only if $(x_1,x_2)$ is of the form \[(\lfloor k\phi \rfloor, \lfloor k\phi^2 \rfloor)\quad\text{or}\quad (\lfloor k\phi^2 \rfloor, \lfloor k\phi \rfloor)\] for some integer $k\ge 0$, where $\phi = (\sqrt{5}+1)/2$ is the golden ratio, and $\lfloor x \rfloor$ denotes the largest integer that is not greater than $x$. 
\end{theorem}

A natural extension of the classic Wythoff's game is the $n$-dimensional Wythoff's game with move vectors $\{(1,0,\dots,0), (0,1,0,\dots,0),\dots,(0,\dots,0,1),(1,1,\dots,1)\}$.
Aggarwal, Geller, Sadhuka and Yu \cite{AGSY16} conjectured that the set of $\mathcal{P}$-positions of $3$-dimensional Wythoff's game is related to the classic fractal---Sierpinski sponge, i.e. the compact set $K\subset \mathbb{R}^3$ satisfying 
\[K=\frac{1}{2}K\cup\frac{1}{2}(K+(1,0,0))\cup\frac{1}{2}(K+(0,1,0))\cup\frac{1}{2}(K+(0,0,1))\]
where $a(A+u):=\{a(x+v)\,:\ x\in A\}$ for $A\subset\mathbb{R}^3$, $u\in\mathbb{R}^3$ and $a\in\mathbb{R}$. For more information on fractal sets, see \cite{Fal14, Ma08}.
\begin{conjecture}[Siperpinski sponge conjecture \cite{AGSY16}]\label{conj:1}
	The set of $\mathcal{P}$-positions of $3$-dimensional Wythoff's game with move vectors $\{(1, 0, 0), (0, 1, 0), (0, 0, 1), (1, 1, 1)\}$ generates the Sierpinski sponge.
\end{conjecture}

We solve conjecture \ref{conj:1} by giving the following more general result. 
\begin{theorem}\label{thm:1}
	Let $n\ge 3$ be an odd number. The set
	\begin{equation*}
	P^{(n)}:=\{(x_1,x_2,\dots,x_n )\in \mathbb{N}^n:x_1 \oplus x_2 \oplus \dots \oplus x_n=0\}
	\end{equation*}
	is all the P-positions of $n$-dimensional Wythoff's game with move vectors  \[V=\{(1,0,\dots,0),(0,1,\dots,0),\dots,(0,0,\dots,1),(1,1,\dots,1)\} ,\]
	where $\oplus$ denotes the bitwise exclusive OR (i.e. the Nim-sum).
\end{theorem}

Theorem \ref{thm:1} explicitly describes the $\mathcal{P}$-positions of $n$-dimensional Wythoff's game with move vectors $V$. When $n=3$, we see $V = \{(1,0,0),(0,1,0),(0,0,1),(1,1,1)\}$. In this case $P^{(3)}$ is the set of all $\mathcal{P}$-positions (in $\mathbb{R}^3$) which is unbounded. To visualize $P^{(n)}$, we introduce its bounded version. For any $m\ge 0$, write 
\[P_m^{(n)}=\left\{(x_1,x_2,\dots,x_n )\in \mathbb{N}^n : x_i<2^m \text{ for }1\leq i\leq n \text{ and }x_1\oplus x_2\oplus\dots\oplus x_n =0 \right\}.\]
Fig.~\ref{fig.1} illustrates $P^{(3)}_6$ in two different angles. It is clear that \[P_m^{(n)}\subset P_{m+1}^{(n)}\subset P^{(n)}\quad \text{and}\quad P^{(n)}=\lim\limits_{m\to\infty}P_m^{(n)}.\]

As shown in Fig.~\ref{fig.1}, there is a remarkable resemblance between $P^{(3)}_6$ and the Sierpinski sponge. Although $P^{(3)}$ shares a similar self-similarity as the Sierpinski sponge, the set $P^{(3)}$, which is discrete and unbounded, is essentially different from the Sierpinski sponge. So the next issue will be finding the relationship between them. 
Since $P_m^{(n)}$ is unbounded (in $\mathbb{R}^{n}$), we consider its scaled copy $P_m^{(n)}/2^m$ which is increasing as $m$ increases (see Section 3). In such a way, the set $\mathcal{P}^{(n)}:=\lim_{m\to\infty}P^{(n)}_m/2^{m}$ is a bounded version of $P^{(n)}$. By using the nested structure of $P_m^{(n)}$ (see Lemma \ref{3.3.3n} in section 3), we obtain the nested structure of $\mathcal{P}^{(n)}$ in the following result.  

\begin{figure}[htbp]
	\centering
	\subfigure[Angle 1]{
		\includegraphics[width=0.49\textwidth]{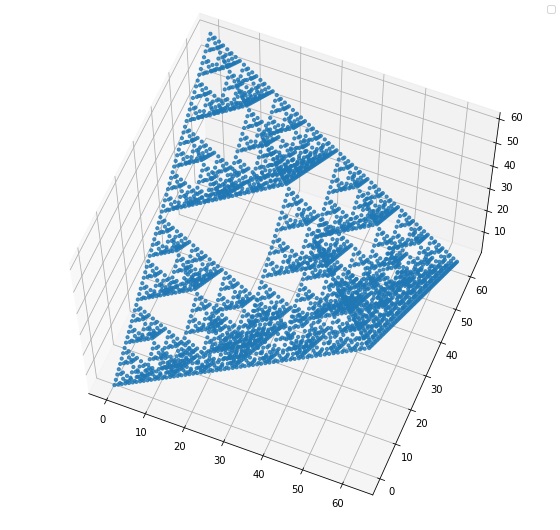} 
	}
	\subfigure[Angle 2]{
		\includegraphics[width=0.46\textwidth]{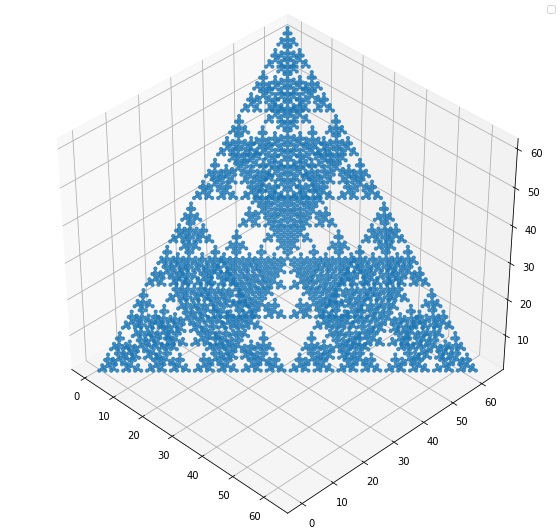} 
	}
	\caption{$\mathcal{P}$-positions of 3-dimensional Wythoff's game ($x_i<64$)}
	\label{fig.1}
\end{figure}

\begin{theorem}\label{thm:2}
	For every odd number $n\ge 3$, 
	\[\mathcal{P}^{(n)}=\bigcup_{v\in T^{(n)}}\frac{1}{2}\left(v+\mathcal{P}^{(n)}\right)\]
	where $T^{(n)}=\{(x_1,x_2,\dots,x_n)\in \{0,1\}^n :x_1\oplus x_2\oplus\dots\oplus x_n =0 \}$.
\end{theorem}

Note that the mapping $g_{v}(x)=\frac{1}{2}(x+v)$ is a contraction on $\mathbb{R}^n$ for all $v\in T^{(n)}$. Hutchinson \cite{Hut81} showed that there is a unique closed bounded set $E$ satisfying $E=\bigcup_{v\in T^{(n)}}g_v(E)$. Combining this fact and Theorem \ref{thm:2}, we have the following result which answers Conjecture \ref{conj:1}.
\begin{corollary}
	The closure of $\mathcal{P}^{(3)}$ is the Sierpinski sponge $K$.
\end{corollary}

The paper is organized as follows. In Section 2, we characterize the $\mathcal{P}$-positions of $n$-dimensional Wythoff's game with move vectors $V$ for all odd numbers $n$. Then we prove Theorem \ref{thm:1}. In Section 3, we study the geometric property of $P^{(n)}$ and prove Theorem \ref{thm:2}.

\section{\texorpdfstring{$\mathcal{P}$}{P}-positions of \texorpdfstring{$n$}{n}-dimensional Wythoff's game}

In this section, we give a quite simple description for $\mathcal{P}$-positions in $n$-dimensional Wythoff's game with move vectors $V=\{(1,0,\dots,0),(0,1,\dots,0),\dots,(0,0,\dots,1),(1,1,\dots,1)\}$. 

The following lemma, which is a consequence of \cite[Proposition 1]{DFNR10}, gives a criterion telling when the set of positions is the set of all $\mathcal{P}$-positions of a game.
\begin{lemma}[see {\cite[Proposition 1]{DFNR10}}]\label{2.1.4}
	Suppose that $S$ is the set of all possible positions of $n$-dimensional Wythoff's game with move vectors $V'$, and $P$ is a subset of $S$. Then $P$ is the set of all $\mathcal{P}$-positions if $P$ satisfies: 
	\begin{enumerate}
		\item[(i)] for all $x\in P$, $x-kv\in S\backslash P$ for all $k\ge 1$ and $v\in V'$ providing $x-kv\in\mathbb{N}^n$; 
		\item[(ii)] for all $x\in S\backslash P$, there exist $k\ge 1$ and $v\in V'$ such that $x-kv\in P$.
	\end{enumerate} 
\end{lemma}


\subsection{Proof of Theorem \ref{thm:1}}
Let $n\ge 3$ be an odd number. Now we characterize the $\mathcal{P}$-positions of $n$-dimensional Wythoff's game with move vectors \[V=\{(1,0,\dots,0),(0,1,\dots,0),\dots,(0,0,\dots,1),(1,1,\dots,1)\}.\] Recall that $\oplus$ denotes bitwise exclusive OR. We restate the result as below.
\begin{theorem1}
	The set of all $\mathcal{P}$-positions of $n$-dimensional Wythoff's game with move vectors $V$ is 
	$
	P^{(n)}=\{(x_1,x_2,\dots,x_n )\in \mathbb{N}^n:x_1 \oplus x_2 \oplus \dots \oplus x_n=0\}
	$.
\end{theorem1}

\begin{proof}
	Denote by $S^{(n)}$ the set of all possible positions of $n$-dimensional Wythoff's game with move vectors $V$. By Lemma \ref{2.1.4}, we only need to show that $P^{(n)}$ satisfies the conditions (i) and (ii) in Lemma \ref{2.1.4}. Namely, every position in $P^{(n)}$ will be changed into a position in $S^{(n)}$ by any legal move, and every position in $S^{(n)}\backslash P^{(n)}$ can be changed into a position in $P^{(n)}$ by some legal move. 
	
	(i) Let $x=(x_1,x_2,\dots,x_n )\in P^{(n) }$. We need to show that $x-tv\in S^{(n)}\backslash P^{(n)}$ for all $t\ge 1$ and $v\in V$ providing $x-tv\in\mathbb{N}^N$. 

	When $v=(1,1,\dots,1)$, we have $x-tv=(x_1-t,x_2-t,\dots,x_n-t)$. Consider the binary representation $t= \sum_{k=0}^{\infty}2^kf_k(t)$ where $f_k (t)=0$ or $1$. Let $\tau =\min\{k\geq 0:f_k(t)=1\} $. It follows that $t\equiv 2^\tau \pmod{ 2^{\tau +1}}$. For any component $x_i =\sum_{k=0}^{\infty}2^kf_k(x_i)$ of $x$, since $x_i-t\ge 0$, we have
	\begin{align}
	f_\tau (x_i )=1 
	\iff  & \sum_{k=0}^{\tau}2^kf_{k}(x_i) \geq 2^\tau \nonumber\\
	\iff  & \sum_{k=0}^{\tau}2^kf_{k}(x_i-t) < 2^\tau \nonumber\\
	\iff  & f_\tau (x_i-t)=0.\label{3-1}
	\end{align}
	Since both $f_\tau (x_i )$ and $f_\tau (x_i-t)$ can only be 0 or 1, Eq.~\eqref{3-1} actually shows that $f_\tau (x_i )\oplus f_\tau (x_i-t)=1$. By the definition of $P^{(n)}$, we see $x_1\oplus x_2\oplus \dots \oplus x_n=0$. It follows from Lemma \ref{lem:7} that for all $k\ge 0$, 
	\begin{equation}\label{3-1-1}
		f_k (x_1 )\oplus f_k (x_2 )\oplus\dots\oplus f_k (x_n )=0.
	\end{equation}
	Note that $n\ge 3$ is an odd number. By Eq.~\eqref{3-1} and Eq.~\eqref{3-1-1},  we have
	\[f_\tau (x_1-t)\oplus f_\tau (x_2-t)\oplus \dots \oplus f_\tau (x_n-t)=1. \]
	Therefore, 
	\[(x_1-t)\oplus(x_2-t)\oplus \dots \oplus(x_n-t)\ne 0\]
	which means $x-tv\notin P^{(n)}$.
	
	When $v = (1,0,\dots,0)$, we have $x-tv=(x_1-t,x_2,\dots,x_n)$. Since $x\in P^{(n)}$, for all $k\ge 0$, 
	\begin{equation}\label{3-2}
	f_k (x_1 )\oplus f_k (x_2 )\oplus \dots \oplus f_k (x_n )=0.
	\end{equation} 
	It follows from Eq.~\eqref{3-1} that 
	\begin{equation}\label{3-3}
	f_\tau (x_1 )\oplus f_\tau (x_1-t)=1.
	\end{equation}
	By Eq.~\eqref{3-2} and Eq.~\eqref{3-3}, we see $f_\tau (x_1-t)\oplus f_\tau(x_2 )\oplus \dots \oplus f_\tau (x_n )=1$ and	
	\[(x_1-t)\oplus x_2 \oplus \dots \oplus x_n\ne 0. \]
	So $x-tv\notin P^{(n)}$.

	For the other cases $v\in\{(0,1,\dots,0),\dots,(0,\dots,0,1)\}$, we also see $x-tv\notin P^{(n)}$ by applying the same discussion as in the case $v=(1,0,\dots,0)$. 
	
	(ii) Suppose $x=(x_1,x_2,\dots ,x_n )\in S^{(n)}\backslash P^{(n)}$. 
	By the definition of $P^{(n)}$, we know that $ x_1\oplus x_2\oplus \dots \oplus x_3\ne 0$. Let
	\[k' = \max \{ k\geq 0: f_k(x_1 )\oplus f_k (x_2 )\oplus \dots \oplus f_k (x_n )=1\}.\]
	Since $f_k (x_i )=0$ or $1$, there exists an $i$ such that $f_{k'} (x_i )=1$. Without loss of generality, we suppose $i=1$. Let 
	\[b_k=
	\begin{cases}
	f_k(x_1) &\text{if }k>k',\\
	f_k(x_2)\oplus f_k(x_3)\oplus \dots \oplus f_k(x_n) &\text{if }k\leq k'.
	\end{cases}\]
	Obviously, $b_k=0$ or $1$ for all $k\geq 0$. Set $\widetilde{x}_1=\sum_{k=0}^{\infty}2^kb_k$. Then
	\[b_k\oplus f_k (x_2 )\oplus \dots \oplus f_k (x_n )=0,
	\]
	which implies
	\[\widetilde{x}_1\oplus x_2\oplus \dots \oplus x_n=0.\]
	Note that $f_{k'}(x_1 )=1$ and $f_{k'}(x_1)\oplus b_{k'}=1$. We have $b_{k'}=0$. Thus
	\begin{align*}
		\widetilde{x}_1 & = \sum_{k=0}^{k'-1}2^k b_k +\sum_{k=k'+1}^{\infty}2^kb_k\\
		& = \sum_{k=0}^{k'-1}2^k b_k +\sum_{k=k'+1}^{\infty}2^kf_k(x_1)\\
		& \leq 2^{k'}f_{k'}(x_1) + \sum_{k=k'+1}^{\infty}2^kf_k(x_1) \leq x_1.
	\end{align*}
	Letting $t=x_1-\widetilde{x}_1$ and $v=(1,0,\dots ,0)$, we obtain that  $x-tv=(\widetilde{x}_1,x_2,\dots,x_n)\in P^{(n)}$.
\end{proof}

\section{\texorpdfstring{$n$}{n}-Dimensional Discrete Sierpinski Sponge}
In this section, we investigate the structure of the set \[	P^{(n)}:=\{(x_1,x_2,\dots,x_n )\in \mathbb{N}^n:x_1 \oplus x_2 \oplus \dots \oplus x_n=0\}\] and find the relationship between $P^{(n)}$ and the Sierpinski sponge. Recall that for all $m\ge 0$, 
\[P_m^{(n)}=\left\{(x_1,x_2,\dots,x_n )\in \mathbb{N}^n : x_i<2^m \text{ for }1\leq i\leq n \text{ and }x_1\oplus x_2\oplus\dots\oplus x_n =0 \right\}.\]
We call the set $P_m^{(n)}$ an ($n$-dimensional) \emph{discrete Sierpinski sponge}. Obviously, $P_m^{(n)}$ is a subset of $P^{(n)}$ and $P^{(n)} = \bigcup_{m=0}^\infty P_m^{(n)}$. So in the following, we shall find a way to obtain the Sierpinski sponge from the $n$-dimensional discrete Sierpinski sponge $P_m^{(n)}$. 

To study the structure of $P^{(n)}$, we need an auxiliary lemma on the bitwise exclusive OR $\oplus$. Recall that for $t\in\mathbb{N}$, its binary representation is denoted by \[t=\sum_{k=0}^{\infty}2^mf_{k}(t)\] where $f_{k}(t)=0$ or $1$ for all $k\ge 0$.
\begin{lemma}\label{lem:7}
	Let $\ell\ge 2$ be an integer and $(x_1,x_2,\dots,x_{\ell})\in \mathbb{N}^{\ell}$. Then 
	\begin{enumerate}
		\item[(i)] $f_k(x_1\oplus x_2\oplus \dots\oplus x_{\ell}) = f_{k}(x_1)\oplus f_{k}(x_2)\oplus\dots\oplus f_{k}(x_{\ell})$ for all $k\ge 0$;
		\item[(ii)] $x_1\oplus x_2\oplus \dots\oplus x_{\ell}=0$ if and only if $f_{k}(x_1)\oplus f_{k}(x_2)\oplus\dots\oplus f_{k}(x_{\ell})=0$ for all $k\ge 0$.  
	\end{enumerate} 
\end{lemma}
\begin{proof}
	(i) When $\ell=2$, the result follows from the definition of $\oplus$. In general, set $y=x_1\oplus x_{2}\oplus \dots\oplus x_{\ell-1}$. Noting that 
	\begin{align*}
		f_k(x_1\oplus x_2\oplus \dots\oplus x_{\ell}) & = f_k(y\oplus x_{\ell}) = f_k(y)\oplus f_{k}(x_{\ell}),
	\end{align*}
	the result follows from induction on $\ell$.

	(ii) When $\ell=2$, the definition of $\oplus$ gives 
	\begin{align*}
		x_1\oplus x_2=0 \iff & f_{k}(x_1)=f_{k}(x_2) \text{ for all }k\ge 0,\\
		\iff & f_{k}(x_1)\oplus f_{k}(x_2)=0 \text{ for all }k\ge 0.
	\end{align*}
	In general, set $y=x_1\oplus x_{2}\oplus \dots\oplus x_{\ell-1}$. Then $y\oplus x_{\ell}=0$ if and only if $f_{k}(y)\oplus f_{k}(x_{\ell})=0$ for all $k\ge 0$. Now use Lemma \ref{lem:7}(i), and the result follows.
\end{proof}

The next result shows the nested structure of $P_m^{(n)}$. 
\begin{lemma}\label{3.3.3n}
	For all $m\geq 0$, 
	\begin{equation}\label{3-4n}
	P_{m+1}^{(n)}=\bigcup_{v\in T^{(n)}} (2^m v+P_m^{(n)})
	\end{equation}
	where $T^{(n)}=\{(x_1,x_2,\dots,x_n)\in \{0,1\}^n :x_1\oplus x_2\oplus\dots\oplus x_n =0 \}$.
\end{lemma}

\begin{proof}
	Let $x=(x_1,x_2,\dots,x_n)\in P_{m+1}^{(n)}$. Then by Lemma \ref{lem:7}(ii), we have for all $k\ge 0$, $f_k(x_1)\oplus f_k(x_2)\oplus \dots \oplus f_k(x_n) = 0$, which yields that 
	\[v':=(f_m(x_1), f_m(x_2), \dots , f_m(x_n)) \in T^{(n)}.\]
	Let $\widetilde{x} = (\widetilde{x}_1,\widetilde{x}_2,\dots,\widetilde{x}_n) = x-2^mv'$. It is clear that 
	\[f_m (\widetilde{x}_1) = f_m (\widetilde{x}_2) = \dots =f_m (\widetilde{x}_n) = 0\]
	and for $0\leq k<m$, 
	\begin{equation}\label{eq:6}
		f_k (\widetilde{x}_1)\oplus f_k (\widetilde{x}_2)\oplus \dots \oplus f_k (\widetilde{x}_n) = 0.
	\end{equation}
	Note that $\widetilde{x}_i = \sum_{k=1}^{m-1}2^kf_k(\widetilde{x}_i) < 2^m$ for all $i$. By Lemma \ref{lem:7}(ii), Eq.~\eqref{eq:6} implies $\widetilde{x}_1\oplus \widetilde{x}_2\oplus\dots\oplus\widetilde{x}_n=0$. Hence $\widetilde{x}\in P_{m}^{(n)}$. In other words, $x\in (2^m v' + P_{m}^{(n)})$. This shows $P_{m+1}^{(n)}\subset\bigcup_{v\in T^{(n)}} (2^m v+P_m^{(n)})$.
	
	Let $v''\in T^{(n)}$ and $x=(x_1,x_2,\dots,x_n)\in (2^m v'' + P_{m}^{(n)})$. Then there exists $y\in P_m^{(n)}$ such that $x_i=2^mv''_i+y_i$ for all $i$. Since $y_i<2^m$, we have $v_i''=f_k(x_i)$ for all $i$. Namely,
	\[v'' = (f_m(x_1), f_m(x_2), \dots , f_m(x_n)).\]
	Therefore, 
	\begin{equation}\label{eq:7}
		f_m(x_1)\oplus f_m(x_2)\oplus \dots \oplus f_m(x_n) = 0.
	\end{equation}
	Since $f_k(x_i)=f_{k}(y_i)$ for $i=1,\dots,n$, we have 
	\begin{equation}\label{eq:8}
		f_k(x_1)\oplus f_k(x_2)\oplus \dots \oplus f_k(x_n) = 0
	\end{equation}
	for $0\leq k<m$.
	Note that $\widetilde{x}_i = \sum_{k=1}^{m}2^kf_k(\widetilde{x}_i) < 2^{m+1}$. By Lemma \ref{lem:7}(ii), Eq.~\eqref{eq:7} and Eq.~\eqref{eq:8} yield $x\in P_{m+1}^{(n)}$. This implies $P_{m+1}^{(n)}\supset\bigcup_{v\in T^{(n)}} (2^m v+P_m^{(n)})$ and the desired result follows.
\end{proof}

In fact, we can see that for all $v,u\in T^{(n)}$ with $v \ne u$, 
\[(2^m v+P_m^{(n)})\cap(2^m u+P_m^{(n)}) = \emptyset\]
and 
\[\mathrm{Card}\left(T^{(n)}\right) = \sum_{i=0}^{(n-1)/2} \binom{n}{2i} = 2^{n-1}. \]
Thus Lemma \ref{3.3.3n} actually shows that $P_{m+1}^{(n)}$ is composed by $2^{n-1}$ translated copies of $P_m^{(n)}$; see also Fig. \ref{fig.1}. 

\subsection{Proof of Theorem \ref{thm:2}}
To obtain a bounded version of $P^{(n)}$, we consider the scaled copy $2^{-m}P^{(n)}_m$. According to Lemma \ref{lem:7}, the set $P^{n}_m$ has an equivalent definition as follows:
\begin{align}
	P_m^{(n)}=\Big\{(x_1,x_2,\dots,x_n)\in\mathbb{N}^n\,:\ & x_i<2^m \text{ for }1\leq i\leq n \text{ and } \nonumber\\
	& f_{k}(x_1)\oplus f_{k}(x_2)\oplus\dots\oplus f_{k}(x_n) =0 \text{ for all } k\ge 0\Big\}. \label{eq:9}
\end{align}
For any $y\in (1/2^m )P_m^{(n)}$, there exists $x\in P_m^{(n)}$ such that $y_i=(1/2^m)x_i$ for $i=1,2,\dots,n$. Let $x_i=\sum_{k=0}^{m}2^kf_{k}(x)$. Then 
\begin{align*}
y_i & =(1/2^m)x_i = (1/2^m )\sum_{k=0}^{m-1}2^kf_k(x_i)
= \sum_{k=0}^{m-1}2^{k-m}f_{k}(x_i) = \sum_{k=1}^{m}2^{-k}f_{m-k}(x_i).
\end{align*}
Since $x\in P_m^{(n)}$, from Eq.~\eqref{eq:9}, we obtain that for $1\leq k\leq m$,
\begin{align*}
	0 & = f_{m-k}(x_1)\oplus f_{m-k}(x_2)\oplus\dots\oplus f_{m-k}(x_n) \\
	  & = f_{m-k}(2^{m}y_1)\oplus f_{m-k}(2^{m}y_2)\oplus\dots\oplus f_{m-k}(2^{m}y_n).
\end{align*}
Now we have for $m\ge 1$, 
\begin{align*}
	(1/2^m )P_m^{(n)}=\Big\{y=(y_1,y_2,\dots,y_n) \in [0,1)^n\,:\ & 2^my\in\mathbb{N}^n \text{ and for }0\leq k\leq m-1,\\
	& f_{k}(2^{m}y_1)\oplus f_{k}(2^{m}y_2)\oplus\dots\oplus f_{k}(2^{m}y_n) =0\Big\}.
\end{align*}
Apparently, $(1/2^0)P_0^{(n)}=\{(0,0,\dots,0)\}\subset (1/2)P_{1}^{(n)}$. Moreover, from the above, we can see that $(1/2^m)P_m^{(n)}\subset (1/2^{m+1})P_{m+1}^{(n)}$ for all $m\ge 1$. Then \[\mathcal{P}^{(n)}=\lim\limits_{m\to\infty}(1/2^m )P_m^{(n)}=\bigcup_{m=0}^\infty(1/2^m)P_m^{(n)}.\]
We obtain a bounded set which does not depend on the scale $m$. Further, 
\begin{align*}
	\mathcal{P}^{(n)}=\Bigg\{x=(x_1,x_2,\dots,x_n)\in [0,1)^n\,:\ & \exists\, m \text{ s.t. } 2^mx\in\mathbb{N}^n\text{ and for }0\leq k \leq m-1 \\
	& f_{k}(2^{m}x_1)\oplus f_{k}(2^{m}x_2)\oplus\dots\oplus f_{k}(2^{m}x_n) =0\Bigg\}.
\end{align*}

Now we are able to prove our second main result. 
\begin{theorem2}
	For every odd number $n\ge 3$, 
	\begin{equation}\label{3-6n}
		\mathcal{P}^{(n)}=\bigcup_{v\in T^{(n)}} \frac{1}{2} (v+\mathcal{P}^{(n)})
	\end{equation}
	where $T^{(n)}=\{(x_1,x_2,\dots,x_n)\in \{0,1\}^n :x_1\oplus x_2\oplus\dots\oplus x_n =0 \}$.
\end{theorem2}
\begin{proof}
	Since $\mathcal{P}^{(n)}=\bigcup_{m=1}^\infty(1/2^m)P_m^{(n)}$, by Lemma \ref{3.3.3n} we have
	\begin{align*}
		\mathcal{P}^{(n)} =& \bigcup_{m=1}^\infty \left( \frac{1}{2^m} \bigcup_{v\in T^{(n)}}(2^{m-1}v+P_{m-1}^{(n)}) \right) \\
		=& \bigcup_{m=1}^\infty   \bigcup_{v\in T^{(n)}}\left(\frac{1}{2}v+ \frac{1}{2^m}P_{m-1}^{(n)}\right) \\
		=& \bigcup_{v\in T^{(n)}}  \bigcup_{m=1}^\infty \left(\frac{1}{2}v+ \frac{1}{2^m}P_{m-1}^{(n)} \right) \\
		=& \bigcup_{v\in T^{(n)}} \frac{1}{2} \left(v+ \bigcup_{m=1}^\infty ( 1/2^{m-1})P_{m-1}^{(n)} \right)\\
		=& \bigcup_{v\in T^{(n)}} \frac{1}{2} (v+\mathcal{P}^{(n)})
	\end{align*}
	which is Eq.~\eqref{3-6n}.
\end{proof}

When $n=3$, Eq.\eqref{3-6n} actually means
\[\mathcal{P}^{(n)} = \frac{1}{2}\mathcal{P}^{(n)} \cup \frac{1}{2}((1,1,0)+\mathcal{P}^{(n)}) \cup \frac{1}{2}((1,0,1)+\mathcal{P}^{(n)}) \cup \frac{1}{2}((0,1,1)+\mathcal{P}^{(n)})\]
which has the same self-similarity as the Sierpinski sponge. $\mathcal{P}^{(n)}$ obviously contains only rational points and is not necessarily compact. However, the closure of $P^{(n)}$ (in $\mathbb{R}^n$) is compact. Thus $\overline{\mathcal{P}^{(n)}}$ is the attractor of the iteration function system $\{g_v(x):=\frac{1}{2}(v+x)\}_{v\in T^{(n)}}$. 

\medskip
We end the section by giving a description of $\overline{\mathcal{P}^{(n)}}$.

\begin{proposition}\label{3.3.7n}
	Let $n\ge 3$ be an odd number. Then $\overline{\mathcal{P}^{(n)}}=\mathcal{Q}^{(n)}$ where 
	\begin{multline}\label{3-7n}
		\mathcal{Q}^{(n)}=\Bigg\{x\in [0,1]^n : 
		\text{there exist $0$-$1$ sequences } (a_{k}(x_i))_{k\geq 1}\text{ such that }\\
		x_i = \sum_{k=1}^{\infty}a_{k}(x_i)2^{-k}
		\text{ and }a_{k}(x_1)\oplus a_{k}(x_2)\oplus \dots \oplus a_{k}(x_n)=0 \text{ for all }k \Bigg\}.
	\end{multline}
\end{proposition}

\begin{proof}	
	Let $x=(x_1,x_2,\dots ,x_n )\in \mathcal{Q}^{(n)}$. For all $m\ge 1$, write  
	$x_i|_m =\sum_{k=1}^{m}2^{-k}a_{k}(x_i)$ and let
	$x|_m=(x_1|_m,x_2|_m,\dots ,x_n|_m)$.  
	Then $\{x|_m\}_{m\geq 1}\subset \mathcal{P}^{(n)}$ and it is clear that $ \big|x_i-x_i|_m\big|\to 0$ as $m\to \infty$. Thus $x$ is an accumulation point of $\mathcal{P}^{(n)}$. So $\mathcal{P}^{(n)}\subset\mathcal{Q}^{(n)}\subset \overline{\mathcal{P}^{(n)}}$.
	
	Now we show that $\mathcal{Q}^{(n)}$ is sequentially compact. Suppose $\{\mathbf{y}_j\}_{j\geq 1}$ is a sequence in $\mathcal{Q}^{(n)}$ where $\mathbf{y}_j = (y_{j,1},y_{j,2},\dots,y_{j,n})$. Suppose $y_{j,i} = \sum_{k=1}^{\infty}2^{-k}a_{k}(y_{j,i})$ with $a_{k}(y_{j,i})=0$ or $1$, and for all $j\ge 1$, $k\ge 1$, 
	\begin{equation}\label{eq:12}
		a_{k}(y_{j,1})\oplus a_{k}(y_{j,2})\oplus \dots \oplus a_{k}(y_{j,n})=0.
	\end{equation}
	Let $b_{k,j}:=\left(a_{k}(y_{j,1}),a_{k}(y_{j,2}),\dots,a_{k}(y_{j,n})\right)\in\{0,1\}^n$. Then $(b_{k,j})_{k\ge 1,j\ge 1}$ is a double sequence taking values in $\{0,1\}^n$. Since $\{0,1\}^n$ has only $2^n$ different values, there exist $\mathbf{c}_1=(c_{1,1},$ $c_{1,2},\dots,c_{1,n})\in\{0,1\}^n$ and an infinite subset $N_1\subset \mathbb{N}$ such that $b_{1,j}=\mathbf{c}_1$ for all $j\in N_1$. Write the first term of $(\mathbf{y}_{j})_{j\in N_1}$ by $\mathbf{z}_1$. For the same reason, there exist $\mathbf{c}_2\in\{0,1\}^n$ and an infinite subset $N_2\subset N_1$ such that $b_{2,j}=\mathbf{c}_2$ for all $j\in N_2$. Write the second term of $(\mathbf{y}_{j})_{j\in N_2}$ by $\mathbf{z}_2$. Repeating the previous procedure, we find a sequence $(\mathbf{c}_k)_{k\ge 1}$ on $\{0,1\}^n$ and a subsequence $(\mathbf{z}_k)_{k\ge 1}$ of $(\mathbf{y}_j)_{j\ge 1}$. Moreover, $(\mathbf{z}_k)_{k\ge 1}$ converges to $\mathbf{z}=(z_1,z_2,\dots,z_n)$ where $z_i=\sum_{k=1}^{\infty}c_{k,i}2^{-k}$. By Eq.~\eqref{eq:12}, we have for all $k\ge 1$, 
	\[c_{k,1}\oplus c_{k,2}\oplus \dots\oplus c_{k,n}=0.\]
	So $\mathbf{z}\in\mathcal{Q}^{(n)}$, which yields the desired compactness. Consequently, $\mathcal{Q}^{(n)}$ is closed. Given that $\mathcal{P}^{(n)}\subset\mathcal{Q}^{(n)}\subset \overline{\mathcal{P}^{(n)}}$, we obtain that $\mathcal{Q}^{(n)}$ is the closure of $\mathcal{P}^{(n)}$. 
\end{proof}

\section*{Acknowledgement}
    This work was supported by ``the Fundamental Research Funds for the Central Universities'' (No. 2020ZYGXZR041) and Guangdong Natural Science Foundation (Nos. 2018A030313971, 2018B0303110005).

\end{document}